\documentclass[11pt]{amsart}
\usepackage{amsmath, amssymb, amsthm}
\usepackage{hyperref}
\usepackage{tikz}
\usepackage{subfigure}
\usepackage{ytableau}
\usepackage{array}
\usepackage{amsfonts}
\usepackage{mathtools}
\theoremstyle{plain}
\newtheorem{theorem}{Theorem}[section]
\newtheorem{lemma}[theorem]{Lemma}
\newtheorem{corollary}[theorem]{Corollary}
\newtheorem{proposition}[theorem]{Proposition}

\renewcommand{\Im}{\operatorname{Im}}

\numberwithin{equation}{section}
\begin{document}
	\title[ Asymptotic Formula for Multipartitions
	]
	{Asymptotic Formula for Multipartitions} 
	\author{JAYANTA BARMAN}
	\address{JAYANTA BARMAN\\ Department of Mathematics \\
		Indian Institute of Technology Kharagpur \\
		Kharagpur-721302,  India.} 
	\email{b1999jayanta@gmail.com}
	
	\author{Kamalakshya Mahatab}
	\address{Kamalakshya Mahatab\\ Department of Mathematics \\
		Indian Institute of Technology Kharagpur \\
		Kharagpur-721302,  India.} 
	\email{accessing.infinity@gmail.com, kamalakshya@maths.iitkgp.ac.in}
	
	\subjclass[2020]{11P82, 11N37}
	\keywords{ Partition Function, t-multipartition, Saddle Point Method}
	\begin{abstract} 
		Let $p_t(N)$ denote the number of $t$-multipartitions of a positive integer $N$. In this article, we obtain an asymptotic formula for $p_t(N)$, when $t \ll N^{1 - \epsilon}$, for any $\epsilon > 0$.
	\end{abstract}
	\maketitle
	\section{Introduction}\label{section1} 
	The asymptotic theory of partition functions has its origin in the celebrated work of Hardy and Ramanujan \cite{hardy1918asymptotic}, who introduced the circle method and obtained an asymptotic formula for the ordinary partition function $p(N)$. 
	Although unrestricted partitions form a central object of study, imposing constraints on partitions opens many deep and interesting questions in both arithmetic and combinatorics.
	One classical example is that of $t$-multipartitions: a $t$-multipartition of a positive integer $N$ is a sequence $\left(\mu^{(1)},\mu^{(2)},\ldots,\mu^{(t)}\right)$, where each $\mu^{(j)}$ is an integer partition (possibly empty) for all $j$ such that $\sum_{j=1}^{t}\left|\mu^{(j)}\right|=N$. Here, $\left|\mu^{(j)}\right|$ denotes the sum of the parts of the partition $\mu^{(j)}$. Let $p_{t}(N)$ denote the number of $t$-multipartitions of $N$. Further, the number of $t$-multipartitions of $N$ coincides with the number of $t$-colored partitions of $N$. Equivalently, a $t$-colored partition of $N$ is a partition in which each part may be assigned one of $t$ available colors, with the order of the colors being immaterial. For instance, the $2$-colored partitions of $3$ are $3_1, 3_2, 2_1+1_1, 2_1+1_2, 2_2+1_1,2_2+1_2,1_1+1_1+1_1, 1_2+1_1+1_1, 1_2+1_2+1_1, 1_2+1_2+1_2.$ 
	Since these two notions are equivalent from an enumerative point of view, we may also interpret $p_t(N)$ as the number of $t$-colored partitions of $N$. Various properties of $p_t(N)$, including log-concavity, arithmetic properties, and multiplicative properties, have been studied in \cite{bringmann, bringmann1, chern}. The generating function for $t$-multipartitions, together with their congruence and other arithmetic properties, has been investigated extensively in \cite{andrews, atkin, chen2014}. 
	
	James and Kerber \cite[Corollary 4.4.4]{MR644144} showed that $p_t(N)$ can be expressed as a sum of products of partition functions, given by
	\begin{equation*}
		\sum_{\substack{N=n_{1}+n_{2}+\cdots+n_{t}\\ n_{j}\ge 0}} p(n_{1})p(n_{2})\cdots p(n_{t}).
	\end{equation*}
	This is equal to the number of irreducible representations of the wreath product $G \wr S_N$, where the group $G$ has $t$ conjugacy classes and $S_N$ is the symmetric group. There are also applications of $p_t(N)$ in the representation theory of Lie algebras \cite{bouwknegt2002,fayers2006weights}. In \cite{bouwknegt2002}, multipartitions play an important role in the study of Durfee systems and their existence. Furthermore, \cite{fayers2006weights} shows that the irreducible representations of the Ariki--Koike algebra are naturally indexed by $t$-multipartitions of $N$. Multipartitions are also important in algebraic geometry. For a smooth projective surface $S$, let $S^{[N]}$ be the Hilbert scheme of $N$ points, which is a smooth projective variety of dimension $2N.$ Let $\chi(S)$ and $\chi(S^{[N]})$ be the topological Euler characteristic of $S$ and $S^{[N]}$, respectively. Then (see \cite[equation $(2)$, Theorem 0.1]{gottsche}),
	\begin{align*}
		\chi\left(S^{[N]}\right)=p_{\chi(S)}(N).
	\end{align*}
	For applications of multipartitions to gauge theory and random partitions, see \cite{nekrasov}. 

	In \cite{murty2013partition}, Murty applied the Laplace saddle point method to derive an asymptotic formula for $p_t(N)$ for fixed $t$ without providing any error term. This result also follows from Meinardus's formula \cite{meinardus1953}, where special functions play an essential role. In this article, we extend this analysis to the range $t \ll N^{1-\epsilon}$ for any $\epsilon>0$, instead of a fixed $t$, using the saddle point method. This method also has been successfully used to obtain asymptotic formulas for partition functions and related combinatorial sequences \cite{barman2025lower, barman2026asymptotic, tyler2026asymptotics}. 
	
	Now to state our main theorem and give an idea of its proof, we recall some identities involving $p_t(N)$. 
	The generating function for $p_{t}(N)$ is given by
	\begin{equation*}
		G(q)=\prod_{N=1}^\infty \left(1-q^{N}\right)^{-t}=\sum_{N=0}^{\infty}p_t(N)q^{N}=\left(\prod_{N=1}^\infty \frac{1}{\left(1-q^{N}\right)}\right)^{t}=\frac{q^{\frac{t}{24}}}{\eta(z)^{t}}, 
	\end{equation*}
	where $q=\exp(2\pi iz)$, $z=x+iy$ and $y>0$. The Dedekind eta function $\eta(z)$ is given by 
	\begin{equation*}
		\eta(z)=\exp\left(\frac{\pi iz}{12}\right)\prod_{n=1}^{\infty}(1-\exp(2\pi inz)).  
	\end{equation*}
	We will use the functions $\mu_m$, $m\ge 1$, from \cite{tyler2026asymptotics} throughout this article:
	\begin{equation}\label{eq-muk}
		\mu_{m}(z)=-\frac{z^{m+1}}{2\pi i} \left(\frac{d}{dz}\right)^{m} \log\eta(z).
	\end{equation}
	By the Cauchy integral formula, the $N$-th coefficient of the series is given by
	\begin{equation*}
		p_{t}(N)=\frac{1}{2\pi i}\int_{\gamma}\frac{G(q)}{q^{N+1}}dq, 
	\end{equation*}
	where $\gamma$ is a simple positively-oriented loop around the origin, located entirely in the unit circle. Let $\gamma$ be the contour defined by $q=\exp(2\pi i z)$. The condition $|q|<1$ holds if and only if $y>0$. For a fixed value of $y$, as $x$ varies over any interval of length $1$, the variable $q$ moves along a full circle of radius $e^{-2\pi y}$. Therefore, we can write $p_t(N)$ as
	\begin{align}\label{eq-integral}
		p_{t}(N)&=\int_{-1/2}^{1/2} \exp\left( -2\pi izM\right)a_{t}(z)dx,
	\end{align}
	where 
	\begin{align*}
		a_{t}(z)=\frac{1}{\eta(z)^{t}}\quad\text{and}\quad M=N-\frac{t}{24}.
	\end{align*}	
	
	We now state the main theorem of this article.
	\begin{theorem}\label{theorem-1.2}
		Let $t$ be a positive integer $<N$.
		\newline
		(i) Then there exists a unique solution $y>0$, such that
		\begin{align}\label{eq-1.1}
			y^{2}M+t\mu_{1}(iy)=0.
		\end{align}
		(ii) Corresponding to this value of $y$, the function $p_t(N)$ satisfies the following asymptotic approximation:
		\begin{align*}
			p_t(N)=\frac{y^{\frac{3}{2}}\exp(2\pi My)a_{t}(iy)}{\sqrt{t\mu_{2}(iy)}}\left(1+O\left(\frac{y}{t}\right)\right).    
		\end{align*}
	\end{theorem}
	Since $y$ is not given explicitly, we solve for $y$ in terms of $N$ and $t$ using Theorem~\ref{theorem-1.2}$(i)$ and substitute it in part~$(ii)$. This gives the following asymptotic formula for $p_t(N)$.
	\begin{theorem}\label{theorem-1.1}
		For any $\epsilon > 0$, let $t \ll N^{1 - \epsilon}$ be a positive integer. Then
		\begin{align*}
			p_t(N)=&\left(\frac{t}{24}\right)^{\frac{t+1}{4}}\frac{\exp\left(\frac{2\pi}{\sqrt{6}}\sqrt{t}\sqrt{N-\frac{t}{24}}\right)\exp\left(O\left(\frac{t^{\frac{3}{2}}}{\sqrt{N}}\right)\right)}{\sqrt{2}\left(N-\frac{t}{24}\right)^{\frac{t+3}{4}}}\left(1 + O\left(\frac{1}{\sqrt{Nt}}\right)\right). 
		\end{align*}
	\end{theorem}
	
	In Theorem \ref{theorem-1.1}, if we set $t = 1$, we obtain Rademacher's asymptotic formula \cite{rademacher1938partition} for $p(N)$.
	\begin{corollary}\label{corollary-1.3}
		Let $N$ be a large positive integer. Then  
		\begin{equation*}
			p(N)=\frac{1}{4\sqrt{3}N}\exp\left(\frac{2\pi}{\sqrt{6}}\sqrt{N}\right)\left(1+O\left(N^{-\frac{1}{2}}\right)\right).
		\end{equation*}
	\end{corollary}
	In Theorem \ref{theorem-1.1}, if we take $t$ to be any fixed positive integer, then we obtain the following corollary, including the error term. This result coincides with Murty's result \cite{murty2013partition} for $p_t(N)$.
	\begin{corollary}[{\cite[Theorem 4]{murty2013partition}}]\label{corollary-1.4}
		Let $t$ be any fixed
		positive integer. Then
		\begin{equation*}
			p_{t}(N)=\left(\frac{t}{24}\right)^{\frac{t+1}{4}}\frac{\exp\left(\frac{2\pi}{\sqrt{6}}\sqrt{Nt}\right)}{\sqrt{2}N^{\frac{t+3}{4}}}\left(1+O\left(N^{-\frac{1}{2}}\right)\right).
		\end{equation*}  
	\end{corollary}
	We may compare Corollary~\ref{corollary-1.4} with Meinardus's formula \cite{meinardus1953}:
	\begin{align*}
		p_t(N)=\frac{\exp\left(D'(0)\right)\left(t\Gamma(2)\zeta(2)\right)^{\frac{1-2D(0)}{4}}\exp\left(2\sqrt{Nt\Gamma(2)\zeta(2)}\right)}{\sqrt{4\pi}\,N^{\frac{3-2D(0)}{4}}} \left(1+O\left(N^{-\frac{1}{2}}\right)\right), 
	\end{align*}
	where $D(s)=t\zeta(s)$, and $\zeta(s)$ is the Riemann zeta function. This gives an alternative proof of special values of $\zeta(s)$: 
	$$
	\zeta(0)=-\frac{1}{2}, \qquad
	\zeta'(0)=-\frac{1}{2}\log(2\pi), \qquad
	\zeta(2)=\frac{\pi^2}{6}.
	$$ 
	

	\subsection{\texorpdfstring{Sketch of the proof of Theorem \ref{theorem-1.2}}{}} 
	We use the integral in (\ref{eq-integral}) as a basis for our saddle point analysis. 
	Note that  $y=\Im (z)$ only appears on the right-hand side of (\ref{eq-integral}).
	Hence, we select $y$ appropriately as in (\ref{eq-1.1}) to obtain an asymptotic formula for $p_t(N)$. 
	The Taylor series of $a_t(z)=\frac{1}{\eta(z)^{t}}$ has radius of convergence $<y$, so we divide the integral in (\ref{eq-integral}) into two separate parts:
	\begin{align*}
		|x| < \frac{y}{3} 
		\qquad \text{and} \qquad 
		\frac{y}{3} \le |x| \le \frac{1}{2}.
	\end{align*}
	The integral in (\ref{eq-integral}) may be decomposed as
	\begin{equation}
		\begin{aligned}\label{eq-ptN1}
			p_{t}(N)&=\exp(2\pi My)a_{t}(iy)\int_{-y/3}^{y/3}\exp\left(-2\pi i Mx+2\pi i\frac{1}{2\pi i}\log\frac{a_{t}(z)}{a_{t}(iy)}\right)dx\\
			&+\exp(2\pi My)a_{t}(iy)\int_{\frac{y}{3}\le |x|\le \frac{1}{2}}\exp\left(-2\pi i Mx\right)\frac{a_{t}(z)}{a_{t}(iy)}dx.
		\end{aligned}
	\end{equation}   
	The Taylor expansion of $\log a_t(z)$ around $x=0$ is given by
	\begin{align}\label{eq-tyler1}
		\log a_{t}(z)&=\left(\log a_{t}(iy)+x\frac{d}{dz} \log a_{t}(iy)+\frac{x^{2}}{2!}\frac{d^{2}}{dz^{2}}\log a_{t}(iy)+\cdots\right).
	\end{align}
	From the definition of $\mu_m$, we have
	\begin{equation*}
		\left(\frac{d}{dz}\right)^{m}\log a_{t}(z)=t\frac{2\pi i}{z^{m+1}}\mu_{m}(z).
	\end{equation*}
	So, (\ref{eq-tyler1}) simplifies to
	\begin{align}\label{eq-taylor}
		\frac{1}{2\pi i}\log \frac{a_{t}(z)}{a_{t}(iy)}
		&=t\left(x\frac{\mu_{1}(iy)}{(iy)^{2}}+\frac{x^{2}}{2}\frac{\mu_{2}(iy)}{(iy)^{3}}+\frac{x^{3}}{6}\frac{\mu_{3}(iy)}{(iy)^{4}}+\frac{x^{4}}{24}\frac{\mu_{4}(x^{\prime}+iy)}{(x^{\prime}+iy)^{5}}\right),
	\end{align}
	for some $z^{\prime}=x^{\prime}+iy$, where $x^{\prime}$ is between $0$ and $x$. In the region $|x|<\frac{y}{3}$, we evaluate the integral using the Taylor expansion described above. However, this expansion is not valid in the region $\frac{y}{3} \le |x| \le \frac{1}{2}$. Instead, we establish Lemma \ref{proposition-error}, where we derive an upper bound. Combining these results, Proposition \ref{theorem-main} shows that $|x|<\frac{y}{3}$ contributes to the main term, and $\frac{y}{3}\le |x|\le \frac{1}{2}$ contributes to the error in (\ref{eq-ptN1}).

	\section{Acknowledgments}
	J. Barman is deeply thankful to the University Grants Commission (UGC), India, for their invaluable support through the Fellowship Programme. K. Mahatab is supported by the ARG-MATRICS Programme (grant no. ANRF/ARGM/2025/002540/MTR).
	\section{Preliminaries}
	In this section, we establish several preliminary results that are required to prove the main results of this paper. 	
	For small values of $y > 0$, we apply the functional equation for the Dedekind eta function in the following form.
	
	\begin{lemma}[{\cite[Lemma 3.4]{barman2025lower}}]\label{lemma-2.2}
		Let $0<y<\frac{\sqrt{3}}{2}$, and for each $y$, there exist a $v$ satisfying $1<v<1.01$ such that
		\begin{align*}
			\eta(iy) = y^{-\frac{1}{2}} \exp\left( -\frac{\pi}{12y} - v e^{-\frac{2\pi}{y}} \right).
		\end{align*}
	\end{lemma}	
	The following lemma helps us to compute the integral in (\ref{eq-ptN1}) in the region $\frac{y}{3} \le |x| \le \frac{1}{2}$.
	\begin{lemma}\label{proposition-error}
		If $y\le \frac{1}{1000}$, then 
		\begin{align*}
			\int_{\frac{y}{3}\le |x|\le \frac{1}{2}}\left|\frac{a_t(z)}{a_t(iy)}\right|dx\ &\le (1.05)^{t} \exp\left(-\frac{\pi t}{120y}\right).
		\end{align*}
	\end{lemma}
	\begin{proof}
		From the definition of $a_t(z)$,
		\begin{align}\label{eq-atzaty}
			\left| \frac{a_t(z)}{a_t(iy)} \right| = \left| \frac{\eta(iy)}{\eta(z)} \right|^{t}.
		\end{align}
		We now show that the above expression is sufficiently small for $\frac{y}{3}\le |x|\le \frac{1}{2}$ and $y\le \frac{1}{1000}$.
		Let $\gamma= \begin{pmatrix}
			a & b\\
			c & d
		\end{pmatrix}\in \text{SL}_2(\mathbb{Z}) $ and $\Im \gamma z\ge \frac{\sqrt{3}}{2}$.
		Proceeding as in the proof of Proposition~3.1 of \cite{barman2026asymptotic} up to equation~(3.3), we obtain
		\begin{align*}
			\left|\frac{\eta(iy)}{\eta(z)}\right|&\le \left(y\Im \gamma z\right)^{-\frac{1}{4}}\exp\left(-\frac{\pi}{12y}+\frac{\pi}{12y}y\Im \gamma z+\frac{1}{228}\right)\\
			&=:\alpha^{-\frac{1}{4}}\exp\left(\frac{\pi\alpha}{12y}\right)\exp\left(-\frac{\pi}{12y}+\frac{1}{228}\right),
		\end{align*}
		where $\alpha=y\Im \gamma z=\frac{y^{2}}{(cx+d)^{2}+c^{2}y^{2}}$. The function $\alpha^{-\frac{1}{4}}\exp\left(\frac{\pi\alpha}{12y}\right)$ reaches its maximum when $\alpha$ is maximum. Note that  $\alpha<\frac{y^{2}}{c^{2}y^{2}}\le \frac{1}{4}$, for $c\ge 2$. When $c = 1$ and $d = 0$, we have $\alpha=\frac{y^{2}}{x^{2}+y^{2}}\le\frac{9}{10}$, since $|x| \ge \frac{y}{3}$. On the other hand, since $\Im \gamma z \ge \frac{\sqrt{3}}{2}$, it follows that $\alpha \in\left[\frac{\sqrt{3}}{2}y, \frac{9}{10}\right]$. Therefore,
		\begin{align*}
			\left|\frac{\eta(iy)}{\eta(z)}\right|&\le\left(\frac{10}{9}\right)^{\frac{1}{4}}\exp\left(-\frac{\pi}{120y}+\frac{1}{228}\right)  \\
			&\le 1.05 \exp\left(-\frac{\pi}{120y}\right).
		\end{align*}
		Using the above estimate in \eqref{eq-atzaty}, the proof follows immediately.
	\end{proof}
		The following lemma will be used in the proof of Proposition~\ref{theorem-main}.
	\begin{lemma}[{\cite[Lemma 4.1]{tyler2026asymptotics}}]\label{lemma-tyler}
		Suppose $\vartheta>38$, $|\rho|<\frac{2}{25}$, $\xi \in(-1,1)$ and $|\theta|\le 1$. Then,
		\begin{align*}
			\int_{-1/3}^{1/3}\exp\left(2\pi i\rho w\right)\exp\left(-\pi\vartheta w^{2}\left(1+2i\xi w+\theta 3w^{2}\right)\right)dw&=\frac{1}{\sqrt{\vartheta}}\left(1+\theta\frac{3.45}{\vartheta}\right).
		\end{align*}
	\end{lemma}
	In \cite{tyler2026asymptotics}, Tyler derived the following explicit formulas for $\mu_m(z)$ according to whether the imaginary part of $z$ is large ($y\ge 1$) or small ($0<y<1$). Detailed proofs of the following two lemmas are given in \cite{barman2026asymptotic}(See Proposition~3.5 and Proposition~3.6).
	\begin{lemma}\label{eq-(2.1)}
		For large $\Im z$, we use the following formula:
		\begin{align*}
			\notag
			\mu_{m}(z)&=\sum_{n=1}^{\infty}z^{m+1}(2\pi in)^{m-1}\sigma(n)\exp(2\pi inz)- \begin{cases} 
				\frac{z^{2}}{24} & \mbox{if } m=0,1 \\ 
				0 & \mbox{if } m\ge 2 
			\end{cases} \quad\text{and}\\  
			\mu_{m}^{\prime}(z)&=\sum_{n=1}^{\infty}\left((2\pi inz)^{m+1}+(m+1)(2\pi inz)^{m}\right) \frac{\sigma(n)}{2\pi in}\exp(2\pi inz)-\begin{cases} 
				\frac{z}{12} & \mbox{if } m=0,1 \\ 
				0 & \mbox{if } m\ge 2. 
			\end{cases}\\ 
			\notag   
		\end{align*}
	\end{lemma}
	\begin{lemma}\label{eq-mue}
		For small $\Im z$, we have
		\begin{align*}
			\mu_{m}(z)&=\sum_{n=1}^{\infty}P_{m}\left(\frac{2\pi in}{z}\right)\sigma(n)\exp\left(-\frac{2\pi in}{z}\right)+\frac{(-1)^{m}m!}{24}+\frac{z}{4\pi i}\begin{cases} 
				\log(-iz) & \mbox{if } m=0 \\ 
				(-1)^{m-1}(m-1)! & \mbox{if } m\ge 1, 
			\end{cases} \\
			\text{and}\quad &\mu_{m}^{\prime}(z)=\sum_{n=1}^{\infty}Q_{m}\left(\frac{2\pi in}{z}\right)\frac{\sigma(n)}{2\pi in}\exp\left(-\frac{2\pi in}{z}\right)+\frac{1}{4\pi i}\begin{cases} 
				1+\log (-iz) & \mbox{if } m=0 \\ 
				(-1)^{m-1}(m-1)! & \mbox{if } m\ge 1, 
			\end{cases}  
		\end{align*}
	\end{lemma}
	where $P_{m}(s)$ is the degree $m-1$ polynomial defined by the recurrence $P_{0}(s)=s^{-1}$ and
	$P_{m}(s)=(s-m)P_{m-1}(s)-sP^{\prime}_{m-1}(s)$.  For $m=0,1,2,3,4$, explicit values of $P_m$ and $Q_m$ are given in \cite{tyler2026asymptotics} (see (4.18)). 
	
	Next, we prove the following lemma, which will help us to bound the error of Proposition \ref{theorem-main}.
	\begin{lemma}\label{lemma-mu234}
		Let $|x|<\frac{y}{3}$ and $0<y\le\frac{1}{10}$. Then
		\begin{align*}
			\left|\frac{\mu_{3}(iy)}{\mu_2(iy)}\right|<6\qquad\text{and}  \qquad\left|\frac{\mu_{4}(z)}{\mu_2(iy)}\right|<36.   
		\end{align*}
	\end{lemma}
	\begin{proof}
		From the explicit formula for $\mu_2$ given in Lemma \ref{eq-mue}, we obtain
		\begin{equation}\label{eq-mu2}
			\begin{aligned}
				\mu_{2}(iy)&=\sum_{n=1}^{\infty}\left(\frac{2\pi n}{y}-2\right)\sigma(n) \exp\left(-\frac{2\pi n}{y}\right)+\frac{1}{12}-\frac{y}{4\pi}\\
				&\ge \frac{1}{12}-\frac{y}{4\pi}\ge \frac{1}{12}-\frac{1}{40\pi}\ge \frac{1}{16}\quad\left(\text{as $y\le\frac{1}{10}$}\right).
			\end{aligned} 
		\end{equation}
		By Lemma~4.3$(ii)$ of \cite{tyler2026asymptotics}, we have $-\frac{1}{4}<\mu_3(iy)<0$. Thus, \eqref{eq-mu2} gives
		\begin{align*}
			\left|\frac{\mu_{3}(iy)}{\mu_2(iy)}\right|<6. 
		\end{align*}
		Again, from Lemma \ref{eq-mue}, we have
		\begin{equation}\label{eq-mu42}
			\mu_{4}(z)=\sum_{n=1}^{\infty}\left(\left(\frac{2\pi in}{z}\right)^{3}-12\left(\frac{2\pi in}{z}\right)^{2}+36\left(\frac{2\pi in}{z}\right)-24\right)\sigma(n)\exp\left(-\frac{2\pi in}{z}\right)+1-\frac{3z}{2\pi i}.   
		\end{equation}
		For $r\in \mathbb{Z^{+}}$, we have
		\begin{align*}
			& \left|\left(\frac{2\pi in}{z}\right)^{r}\right| =\frac{(2\pi n)^{r}}{(x^{2}+y^{2})^{r/2}}<\frac{(2\pi n)^{r}}{y^{r}} \quad \text{and}\\
			& \left|\exp\left(-\frac{2\pi in}{z}\right)\right|= \left|\exp\left(-\frac{2\pi in(x-iy)}{(x^{2}+y^{2})}\right)\right|=\exp\left(-\frac{2\pi ny}{(x^{2}+y^{2})}\right).
		\end{align*}
		Since $|x|<\frac{y}{3}$, $\left|\exp\left(-\frac{2\pi in}{z}\right)\right|\le \exp\left(-\frac{9\pi n}{5y}\right).$ Again using the bounds $|x|<\frac{y}{3}$ and $0<y\le \frac{1}{10}$, it follows from \eqref{eq-mu42} and the above equation that
		\begin{align*}
			\notag
			|\mu_{4}(z)|
			&\le\sum_{n=1}^{\infty}\left(\frac{(2\pi n)^{3}}{y^{3}}+12\frac{(2\pi n)^{2}}{y^{2}}+36\frac{(2\pi n)}{y}+24\right)\sigma(n)\exp\left(-\frac{9\pi n}{5y}\right)+1+\frac{3y}{2\pi}\sqrt{\frac{10}{9}}\\
			&\le\sum_{n=1}^{\infty}\left((20\pi n)^{3}+12(20\pi n)^{2}+36(20\pi n)+24\right)n^{2}\exp\left(-18\pi n\right)+1+\frac{3}{2\pi}\sqrt{\frac{1}{90}}<2.
		\end{align*}
		Combining (\ref{eq-mu2}), we obtain
		\begin{align*}
			\left|\frac{\mu_{4}(z)}{\mu_2(iy)}\right|<36.     
		\end{align*}
		This completes the proof.
	\end{proof}
	Recall $M=N-\frac{t}{24}$ and let $\theta \in \mathbb{C}$ such that $|\theta|\le 1$. Now, we prove an asymptotic formula for $p_t(N)$. In the proof, the value of $\theta$ may depend on the relevant parameters and may vary from one occurrence to another. In Theorem \ref{theorem-1.2}, we will show that $y$ satisfies $\frac{t\mu_{1}(iy)}{y^{2}}+M=0$. Hence, the restriction \eqref{eq-1.1} on $y$ in the proposition does not affect its generality.
	\begin{proposition}\label{theorem-main} Let $y$ be chosen such that 
		\begin{align}\label{eq-assum}
			\left|\frac{t\mu_{1}(iy)}{y^{2}}+M\right|<\frac{2}{25y},
		\end{align} and assume that $y\le \frac{1}{1000}$. Then
		\begin{align*}
			p_t(N)=\frac{y^{\frac{3}{2}}\exp(2\pi My)a_{t}(iy)}{\sqrt{t\mu_{2}(iy)}}\left(1+\theta\frac{3.48y}{t\mu_2(iy)}\right).    
		\end{align*}
	\end{proposition}
	\begin{proof}
		Recall the integral formula of $p_t(N)$ from (\ref{eq-ptN1}),
		and use Lemma \ref{proposition-error} to write
		\begin{align}\label{eq-p(N,t)}
			p_t(N)&=\exp(2\pi My)a_{t}(iy)\int_{-y/3}^{y/3}\exp\left(-2\pi i Mx+2\pi i\frac{1}{2\pi i}\log\frac{a_{t}(z)}{a_{t}(iy)}\right)dx\\
			\notag
			&+\exp(2\pi My)a_{t}(iy)\left(\theta(1.05)^{t}\exp\left(-\frac{\pi t}{120y}\right)\right).
		\end{align}
		Recall the Taylor expansion of $\log \frac{a_t(z)}{a_t(iy)}$ from (\ref{eq-taylor}) and apply Lemma \ref{lemma-mu234} to bound $\mu_3$ and $\mu_4$ in term of $\mu_2$ to get  
		\begin{align}\label{eq-tayfinal}
			\frac{1}{2\pi i}\log \frac{a_{t}(z)}{a_{t}(iy)}&=tx\frac{\mu_{1}(iy)}{(iy)^{2}}+t\frac{x^{2}}{2}\frac{\mu_{2}(iy)}{(iy)^{3}}\left(1+2i\xi\frac{x}{y}+3\theta\frac{x^{2}}{y^{2}}\right),   
		\end{align}
		where $\xi \in(-1,1)$ and $|\theta|\le 1.$
		\newline
		Let
		\begin{align}\label{eq-alphaerr}
			\vartheta=\frac{t\mu_{2}(iy)}{y} \qquad\text{and} \qquad\rho=y\left(\frac{t\mu_{1}(iy)}{(iy)^{2}}-M\right). 
		\end{align}
		Now the integrand in the first term of \eqref{eq-p(N,t)} further simplifies after substituting \eqref{eq-tayfinal} to
		\begin{align}\label{eq-exppart}
			\notag
			\exp\left(-2\pi i Mx+2\pi i\frac{1}{2\pi i}\log\frac{a_{t}(z)}{a_{t}(iy)}\right)&=\exp\left(\frac{2\pi ix}{y}y\left(\frac{t\mu_1(iy)}{(iy)^{2}}-M\right)-\pi \frac{x^{2}}{y^{2}}\frac{t\mu_{2}(iy)}{y}\left(1+2i\xi \frac{x}{y}+\theta 3\frac{x^{2}}{y^{2}}\right)\right)\\
			&=\exp\left(\frac{2\pi\rho ix}{y}\right)\exp\left(-\pi\vartheta\frac{x^{2}}{y^{2}}\left(1+2i\xi\frac{x}{y}+\theta 3\frac{x^{2}}{y^{2}}\right)\right).
		\end{align}
		Using the lower bound for $\mu_2$ from (\ref{eq-mu2}), we have
		\begin{align*}
			\vartheta\ge \frac{t}{16y}\ge \frac{1000t}{16}>38 \quad\left(\text{as $y\le 1/1000$ and $t\ge 1$}\right).
		\end{align*}
		Further, by (\ref{eq-assum}), $|\rho|<\frac{2}{25}$. By changing the variable $w=x/y$ in \eqref{eq-exppart} and applying Lemma~\ref{lemma-tyler}, we obtain
		\begin{align*}
			\notag
			\int_{-y/3}^{y/3}\exp\left(-2\pi i Mx+2\pi i\frac{1}{2\pi i}\log\frac{a_{t}(z)}{a_{t}(iy)}\right)dx&=\int_{-1/3}^{1/3}\exp\left(2\pi i\rho w\right)\exp\left(-\pi\vartheta w^{2}\left(1+2i\xi w+\theta 3w^{2}\right)\right)y\,dw\\
			&=\frac{y}{\sqrt{\vartheta}}\left(1+\theta\frac{3.45}{\vartheta}\right).
		\end{align*}
		Combining the above equation with (\ref{eq-p(N,t)}), we obtain 
		\begin{align}\label{eq-pNT}
			p_t(N) &=\exp(2\pi My)a_{t}(iy)\frac{y}{\sqrt{\vartheta}}\left(1+\theta\frac{3.45}{\vartheta}+\frac{\sqrt{\vartheta}}{y}\theta(1.05)^{t}\exp\left(-\frac{\pi t}{120y}\right)\right).
		\end{align}
		Since we assume $\frac{1}{y} \ge 1000$, it follows from (\ref{eq-mu2}) that $\mu_2(iy) < \frac{1}{12}$. Hence, using (\ref{eq-alphaerr}), we derive
		\begin{align*}
			\frac{\vartheta^{\frac{3}{2}}}{y}(1.05)^{t}\exp\left(-\frac{\pi t}{120y}\right)&= \frac{t^{\frac{3}{2}}\mu^{\frac{3}{2}}_2(iy)}{y^{\frac{5}{2}}}(1.05)^{t}\exp\left(-\frac{\pi t}{120y}\right)\\
			&\le \frac{(1000)^{\frac{5}{2}}t^{\frac{3}{2}}}{(12)^{\frac{3}{2}}}(1.05)^{t}\exp\left(-\frac{1000\pi t}{120}\right)<0.03.
		\end{align*}
		Using the above estimate in \eqref{eq-pNT}, we obtain
		\begin{align*}
			p_t(N)&=\exp\left(2\pi My\right)a_t(iy)\frac{y}{\sqrt{\vartheta}}\left(1+\theta\frac{3.48}{\vartheta}\right)
			=\frac{y^{\frac{3}{2}}\exp(2\pi My)a_{t}(iy)}{\sqrt{t\mu_{2}(iy)}}\left(1+\theta\frac{3.48y}{t\mu_2(iy)}\right),    
		\end{align*} 
		as $\vartheta=\frac{t\mu_{2}(iy)}{y}.$
		This completes the proof.
	\end{proof}
	\section{Proof of Theorems \ref{theorem-1.2} and \ref{theorem-1.1}}
	We now turn to the proofs of the main results. 
	\begin{proof}[\textbf{\boldmath Proof of Theorem \ref{theorem-1.2}$(i)$}]
		To prove the theorem, we determine the saddle point $y$ by solving $\frac{d}{dz}\left(-2\pi iMz+\log a_t(z)\right)=0$ at $z=iy.$ Hence,
		\begin{align*}
			2\pi iM+t\frac{d}{dz}\log \eta(z)=0.  
		\end{align*}
		Note, $\frac{d}{dz}\log \eta(z)=-\frac{2\pi i}{z^{2}}\mu_1(z)$, as defined in~\eqref{eq-muk}, and setting $z=iy$, we obtain
		\begin{equation*}
			\frac{t\mu_{1}(iy)}{y^{2}}=-M=-N+\frac{t}{24}.   
		\end{equation*}
		Next, we prove that the solution $y>0$ is unique. From the explicit expression of $\mu_1$ for large imaginary part as given in Lemma~\ref{eq-(2.1)}, we have 
		\begin{equation*}
			\mu_{1}(iy)=\frac{y^{2}}{24}-\sum_{n=1}^{\infty}y^{2}\sigma(n)\exp(-2\pi ny), 
		\end{equation*} 
		and for small imaginary part, Lemma~\ref{eq-mue} gives
		\begin{equation}\label{eq-pf12}
			\mu_{1}(iy)=-\frac{1}{24}+\frac{y}{4\pi}+\sum_{n=1}^{\infty}\sigma(n)\exp\left(-\frac{2\pi n}{y}\right).
		\end{equation}
		From the above two expressions, we obtain 
		\begin{align*}
			\lim_{y\to\infty}\frac{t\mu_1(iy)}{y^{2}}=\frac{t}{24} \quad\text{and}\quad \lim_{y\to 0^{+}}\frac{t\mu_1(iy)}{y^{2}}=-\infty \quad\text{(as $y\le \frac{1}{10}$)}.
		\end{align*}
		Hence, there exists some $y>0$ such that 
		\begin{align*}
			\frac{t\mu_1(iy)}{y^{2}}=-N+\frac{t}{24}.  
		\end{align*}  
		By Lemma~\ref{eq-(2.1)} and Lemma \ref{eq-mue}, the functions $\mu_1$ and $\mu_2$ satisfy the relation $$\mu_2(z)=-2\mu_1(z)+z\mu_1^{\prime}(z).$$ Hence,
		\begin{align*}
			\frac{d}{dy}\left(\frac{t\mu_1(iy)}{y^{2}}\right)=\frac{-t\mu_2(iy)}{y^{3}}<0 \quad(\text{from (\ref{eq-mu2})}).    
		\end{align*}
		Since $\frac{t\mu_1(iy)}{y^{2}}$ is a decreasing function, there exists a unique $y>0$ satisfying~(\ref{eq-1.1}).
		
	\end{proof}
	We now prove the second part of the theorem, namely the asymptotic formula for $p_t(N)$.
	\begin{proof}[\textbf{\boldmath Proof of Theorem \ref{theorem-1.2}~(ii)}]
		We proved that $y$ is a unique solution of $\frac{t\mu_1(iy)}{y^{2}}+M=0$, so it is obvious that $\left|\frac{t\mu_1(iy)}{y^{2}}+M\right|\ll\frac{1}{y}$. From Proposition~\ref{theorem-main}, together with the bound $\mu_2(iy)<\frac{1}{12}$ from \eqref{eq-mu2}, we obtain
		\begin{align*}
			p_t(N)=\frac{y^{\frac{3}{2}}\exp(2\pi My)a_{t}(iy)}{\sqrt{t\mu_{2}(iy)}}\left(1+O\left(\frac{y}{t}\right)\right). 
		\end{align*}  
	\end{proof}
	We now prove Theorem~\ref{theorem-1.1} by explicitly writing $y$ in terms of $N$ and $t$ in Theorem~\ref{theorem-1.2}$(i)$. 
	\begin{proof}[\textbf{\boldmath Proof of Theorem \ref{theorem-1.1}}]
		Plugging the formula for $\eta(iy)$ from Lemma~\ref{lemma-2.2} in $a_t(z)=\eta(iy)^{-t}$ in Theorem~\ref{theorem-1.2}$(ii)$, we obtain
		\begin{align}\label{eq-ptN4}
			p_t(N)&=\frac{y^{\frac{t+3}{2}}\exp\left(2\pi My+\frac{\pi t}{12y}\right)E_1(t,y)}{\sqrt{t\mu_{2}(iy)}}\left(1+O\left(\frac{y}{t}\right)\right),
		\end{align}
		where
		\begin{align*}
			E_1(t,y)=\exp\left(tv\exp\left(-\frac{2\pi}{y}\right)\right).
		\end{align*}
		We substitute $\mu_1(iy)$ from (\ref{eq-pf12}) in (\ref{eq-1.1}), and rewrite as a quadratic equation in $y$ as follows:
		\begin{align*}
			&y^{2}M+\frac{ty}{4\pi}-\frac{t}{24}+tA_1=0,  
		\end{align*}
		where
		\begin{align*}
			A_1=\sum_{n=1}^{\infty} \sigma(n) \exp\left( -\frac{2\pi n}{y} \right) . 
		\end{align*}
		We solve $y>0$, from the above equation to obtain 
		\begin{align*}
			y&=-\frac{t}{8\pi M}+\frac{1}{2M}\sqrt{\frac{t^{2}}{16\pi^{2}}+4Mt\left(\frac{1}{24}-A_1\right)}\\
			&=\frac{C_1 t}{M}+\frac{\sqrt{t}}{\sqrt{24M}}+\frac{C_2t^{\frac{3}{2}}}{M^{\frac{3}{2}}}+O\left(\text{max}\left\{A_1,\frac{t^{\frac{5}{2}}}{M^\frac{5}{2}}\right\}\right),
		\end{align*}
		for some constants $C_1$ and $C_2$ independent of $N$ and $t$. 
		\newline
		Using a crude approximation of $y\approx \frac{\sqrt{t}}{\sqrt{24M}}$ in $A_1$, we obtain
		\begin{align}\label{eq-soly}
			y=\frac{C_1 t}{M}+\frac{\sqrt{t}}{\sqrt{24M}}+\frac{C_2t^{\frac{3}{2}}}{M^{\frac{3}{2}}}+O\left(\frac{t^{\frac{5}{2}}}{M^\frac{5}{2}}\right). 
		\end{align}	
		In the above equation, we assume $t\ll N^{1-\epsilon}$ for any $\epsilon>0$.\newline
		We expand $\left(1+\frac{C_1\sqrt{24t}}{\sqrt{M}}+\frac{C_2t\sqrt{24}}{M}+O\left(\frac{t^{2}}{M^{2}}\right)\right)^{-1}$ using the binomial expansion to get
		\begin{align}\label{eq-solyinve}
			y^{-1}=\sqrt{\frac{24M}{t}}-24 C_1-24C_2\left(\sqrt{\frac{t}{M}}\right)+O\left(\sqrt{\frac{t}{M}}\right).    
		\end{align}
		From the explicit form of $\mu_2(iy)$ given in Lemma \ref{eq-mue} and the above $y$, we have
		\begin{align}\label{eq-mub}
			\mu_2(iy)&= \sum_{n=1}^{\infty}\left(\frac{2\pi n}{y}-2\right) \sigma(n)\exp\left(-\frac{2\pi n}{y}\right)+\frac{1}{12}-\frac{y}{4\pi}=\frac{1}{12}\left(1+O\left(\frac{\sqrt{t}}{\sqrt{N}}\right)\right).
		\end{align}
		Recall that $M=N-\frac{t}{24}$.
		Using \eqref{eq-soly} and \eqref{eq-solyinve}, we obtain the following estimates:
		\begin{align*}
			&2\pi My+\frac{\pi t}{12y}= \frac{2\pi}{\sqrt{6}}\sqrt{t\left(N-\frac{t}{24}\right)}+O\left(\frac{t^{\frac{3}{2}}}{\sqrt{N}}\right) \quad\text{and}\\
			&E_1(t,y)=1+O\left(t\exp\left(-C\sqrt{\frac{N}{t}}\right)\right)=1+O\left(\frac{1}{\sqrt{N}}\right)
		\end{align*}
		for some positive constant $C$.
		\newline
		Substituting the value of $y$ as $y=\frac{\sqrt{t}}{\sqrt{24M}}\left(1+O\left(\frac{\sqrt{t}}{\sqrt{N}}\right)\right)$ from \eqref{eq-soly}, the bound for $\mu_2(iy)$ from \eqref{eq-mub}, and the above estimate of $2\pi My +\frac{\pi t}{12y}$ in 
		\eqref{eq-ptN4}, we obtain
		\begin{align*}
			p_t(N)=&\left(\frac{t}{24}\right)^{\frac{t+1}{4}}\frac{\exp\left(\frac{2\pi}{\sqrt{6}}\sqrt{t}\sqrt{N-\frac{t}{24}}\right)\exp\left(O\left(\frac{t^{\frac{3}{2}}}{\sqrt{N}}\right)\right)}{\sqrt{2}\left(N-\frac{t}{24}\right)^{\frac{t+3}{4}}}\left(1 + O\left(\frac{1}{\sqrt{Nt}}\right)\right).   
		\end{align*}
	\end{proof}

	\bibliographystyle{abbrv}
	\bibliography{myrefs4}			
\end{document}